\documentclass[11pt,final]{amsart} % use larger type; default would be 10pt
\usepackage{etex}
\usepackage[utf8]{inputenc} % set input encoding (not needed with XeLaTeX)
\def\bibsection{\section*{References}}

\usepackage{graphicx} % support the \includegraphics command and options

\usepackage{array} % for better arrays (eg matrices) in maths
\usepackage{paralist} % very flexible & customisable lists (eg. enumerate/itemize, etc.)
\usepackage{natbib}
\usepackage{hyperref}
\usepackage{amsmath}
\usepackage{amssymb}
\usepackage{amsthm}
\usepackage{mathrsfs}
\usepackage{enumerate}
\usepackage[all]{xy}
\usepackage{todonotes}
\usetikzlibrary{calc}

\theoremstyle{definition}
\numberwithin{equation}{section}
\newtheorem{thm}[equation]{Theorem}
\newtheorem{prop}[equation]{Proposition}
\newtheorem{cor}[equation]{Corollary}

\newtheorem{lem}[equation]{Lemma}

\newtheorem{remk}[equation]{Remark}

\newcommand{\wt}[1]{\widetilde{#1}}

\newcommand{\vocab}[1]{\textbf{#1}}

\newcommand{\bbr}{\mathbb R}
\newcommand{\bbz}{\mathbb Z}

\newcommand{\bbn}{\mathbb N}

\newcommand{\bbc}{\mathbb C}

\newcommand{\diag}{\mathrm{diag}}

\newcommand{\tr}{\textnormal{tr}}

\newcommand{\supp}{\textnormal{supp}}

\newcommand{\Ad}{\textnormal{ad}}
\newcommand{\mxx}[4]{\left(\begin{array}{cc} #1 & #2\\ #3 & #4\end{array}\right)}

\newcommand{\adj}{\text{adj}}

\begin{document}
\title[The Symmetric $2\times2$ Hypergeometric Matrix Differential Operators]
{The Symmetric $2\times2$ Hypergeometric Matrix Differential Operators}

\author[W. Riley Casper]{W. Riley Casper}
\address{Department of Mathematics, Louisiana State University, Baton Rouge LA}
\email{\href{wcasper1@lsu.edu}{wcasper1@lsu.edu}}


\date{} % Activate to display a given date or no date (if empty),
         % otherwise the current date is printed

\maketitle
\begin{abstract}
We obtain an explicit classification of all $2\times2$ real hypergeometric Bochner pairs, ie. pairs $(W(x),\mathfrak D)$ consisting of a $2\times 2$ real hypergeometric differential operator $\mathfrak D$ and a $2\times 2$ weight matrix satisfying the property that $\mathfrak D$ is symmetric with respect to the matrix-valued inner product defined by $W(x)$.
Furthermore, we obtain a classifying space of hypergeometric Bochner pairs by describing a bijective correspondence between the collection of pairs and an open subset of a real algebraic set whose smooth paths correspond to isospectral deformations of the weight $W(x)$ preserving a bispectral property.
We also relate the hypergeometric Bochner pairs to classical Bochner pairs via noncommutative bispectral Darboux transformations.
\end{abstract}

%\tableofcontents
\section{Introduction}
The $N\times N$ matrix-valued hypergeometric differential equation is the equation
\begin{equation}\label{hypergeometric equation}
(1-x^2)\Psi''(x) + \Psi'(x)A_1(x) + \Psi(x)A_0 = \Lambda\Psi(x)
\end{equation}
and was introduced by Tirao in \cite{tirao} as a natural generalization of the classical hypergeometric differential equation.
In the above, $\Psi(x)$ is an unknown $N\times N$ hypergeometric function, $A_1(x) = A_{11}x + A_{10}$ is a matrix-valued polynomial of degree $1$, and $A_{11},A_{10},A_0,$ and $\Lambda$ are all $N\times N$ complex-valued matrices.
Alternatively, the above expression may be rewritten in terms of a right-acting hypergeometric matrix differential operator as
\begin{equation}\label{hypergeometric operator}
\Psi(x)\cdot \mathfrak D = \Lambda\Psi(x), \ \ \mathfrak D = \partial_x^2(1-x^2)I + \partial_xA_1(x) + A_0.
\end{equation}
An $N\times N$ matrix-valued function $\Psi(x)$ satisfying \eqref{hypergeometric equation} or \eqref{hypergeometric operator} is called a matrix-valued hypergeometric function.

In this paper, we will classify the real hypergeometric matrix Bochner pairs $(W(x),\mathfrak D)$ consisting of a real, $2\times 2$ hypergeometric matrix differential operator $\mathfrak D$ symmetric with respect to a $2\times 2$ weight matrix $W(x)$ in a way made precise below.
In fact, we will show that the set $\mathcal E$ of real hypergeometric matrix Bochner pairs has a natural topological structure in the form of an analytic open subset of a real algebraic set, wherein smooth paths on $\mathcal E$ define isospectral deformations of the weight matrix $W(x)$ in the sense of \cite{adler} preserving a certain bispectral property.
Each pair $(W(x),\mathfrak D)$ will define a sequence of matrix-valued hypergeometric functions which are monic orthogonal matrix polynomials for the corresponding weight.
As such, these orthogonal matrix polynomials are  matrix-valued generalizations of Jacobi polynomials introduced by Gr\"unbaum in \cite{grunbaum2003a,grunbaum2003b}, referred to in the literature as orthogonal matrix polynomials of Jacobi type, and are intimately connected with representation theory \cite{grunbaum2005}.
In particular, they are related to spherical functions on homogeneous spaces obtained from rank $1$ Gelfand pairs and they connect hypergeometric matrix differential operators to the representation theory of Lie groups \cite{heckman,koelink}.

The primary tool for our analysis is the $\Ad$-condition \eqref{ad condition}, a complicated nonlinear relation between the eigenvalues and three-term recursion relation of a sequence of orthogonal matrix polynomials satisfying a second-order matrix differential operator.
The $\Ad$-condition originates from the work of Duistermaat and Gr\"unbaum \cite{duistermaat}, and plays a central role in bispectrality.
For papers specifically relying on an $\Ad$-condition in the noncommutative context of orthogonal matrix polynomials, see \cite{casper2017,casper2018,duran2004,grunbaum2003d,grunbaum2007b}.
Moreover in the scalar case the $\Ad$-condition has been applied directly to the reobtain Bochner's classification result \cite{grunbaum1997}.

Previously, direct application of the $\Ad$-condition to the $N\times N$ matrix Bochner problem for $N>1$ has been stymied by the noncommutativity of the coefficients in the eigenvalues and recurrence relation.
In particular, the equations resulting from the $\Ad$-condition, ie. the $B(n)$, $C(n)$ and $M(n)$-update equations below, are at first blush ugly and intractable.
It is worth noting, however, that the $\Ad$-condition has been successfully applied in the simpler but related context of sequences of matrix polynomials satisfying first order difference and differential equations \cite{grunbaum2003d}.
Even so, in the Bochner case the difference and differential equations are second-order and the relevant equations are very complex, as may be seen by the $B(n)$ and $C(n)$-update equations \eqref{B update equation} and \eqref{C update equation}.

In this paper we will realize this \emph{noncommutativity as an asset} and use it to obtain a complete classification of the $2\times 2$ matrix-valued hypegeometric differential operators.
In particular, the noncommutativity gives us an explicit expression of $M(n)$ in terms of $B(n)$, up to an unknown scalar multiple, as in Proposition \ref{getting M from B}.
By deriving an explicit equation for entries in the $\Ad$-condition, we obtain a system of $42$ high-order polynomial equations in $8$ variables, whose zero locus parametrizes the set of symmetric hypergeometric matrix differential operators.
The polynomials involved are \emph{enormous}, with each polynomial taking up multiple megabytes of memory to express, and painstakingly solved by computer via Gr\"obner-basis type algorithms over the span of several days.

Despite the tremendous computational effort involved in the solution of the polynomial system mentioned in the previous paragraph, the solutions we obtain are simple and beautiful enough to be expressed cleanly in Theorem \ref{classifying space} and \ref{classification theorem} below.
This suggests that while the methods employed herein lead to a difficult mathematical knot, we should be able to sidestep these difficult mathematical computations by other means.
Such a side-step will be especially important for the complete classification of the hypergeometric matrix Bochner pairs for $N>2$, wherein the computational complexity will otherwise greatly increase.
To this end, we suggest several methods of transforming known matrix Bochner pairs into new ones which in an ideal situation might allow us to easily obtain all possible matrix Bochner pairs from a simple starting family of pairs.
These include the noncommutative bispectral Darboux transformations of \cite{casper2017,casper2018}, as well as bispectrality-preserving isospectral deformations of the weight matrix $W(x)$ and in particular those deformations of $W(x)$ arising from isomonodromic deformations of the matrix hypergeometric equation satisfied by $W(x)$ which fix the poles.
As such, the associated deformations are non-Schlesinger in nature and correspond to the presence of resonance in the associated Fuchsian system.
We emphasize here that we do not explicitly develop the connection between isomonodromic deformations and hypergeometric matrix Bochner pairs in this paper, and we intend to develop this connection fully in future work.

\subsection{The classification}
To make our definition of symmetry precise, recall that a weight matrix is a smooth $N\times N$ matrix-valued function $W: \bbr\rightarrow M_N(\bbc)$ which is positive-definite and Hermitian on an interval $\supp(W)\subseteq \bbr$, identically zero outside $\supp(W)$, and which has finite moments, ie. $\int |x|^n W(x)dx <\infty$ for all $n\geq 0$.
Note that in the hypergeometric case, $\supp(W) = (-1,1)$.
Each weight matrix $W(x)$ defines a matrix-valued inner product on the algebra of $N\times N$ matrix-valued polynomials $M_N(\bbc[x])$ given by
$$\langle P(x),Q(x)\rangle_W := \int P(x)W(x)Q(x)^*dx,$$
where here $Q(x)^*$ denotes the Hermitian conjugate of $Q(x)$.
The objective of this paper is to classify the pairs $(W(x),\mathfrak D)$ of $2\times 2$ weight matrices $W(x)$ and $2\times 2$ hypergeometric differential operators $\mathfrak D$ which are symmetric with respect to $W(x)$, ie. which satisfy
$$\langle P(x)\cdot\mathfrak D, Q(x)\rangle_W = \langle P(x),Q(x)\cdot\mathfrak D\rangle_W\ \forall\ P(x),Q(x)\in M_N(\bbc[x]).$$
We call such pairs $(W(x),\mathfrak D)$ hypergeometric matrix Bochner pairs in reference to the 1929 classification result by Bochner in the scalar (ie. $1\times 1$) case \cite{bochner}.
In some previous papers, such as \cite{grunbaum2003a}, these are referred to as ``classic pairs".
We call a matrix Bochner pair real if $W(x)$ and $\mathfrak D$ are both real.

Our classification of hypergeometric matrix Bochner pairs is performed modulo a collection of elementary transformations which trivially create new hypergeometric matrix Bochner pair from a given one.
The collection of all $2\times 2$ hypergeometric matrix Bochner pairs has an action by several obvious groups of automorphisms, described in the list below.
\begin{itemize}
\item translation: $(W(x),\mathfrak D)\mapsto (W(x),\mathfrak D + \alpha I)$ for $\alpha\in \bbr$,
\item conjugation: $(W(x),\mathfrak D)\mapsto (UW(x)U^*,U\mathfrak D U^{-1})$ for $U\in M_2(\bbc)^\times$,
\item reflection: $(W(x),\mathfrak D)\mapsto (W(-x),\sigma(\mathfrak D))$, where here $\sigma$ is the endomorphism of the Weyl algebra induced by $x\mapsto -x$.
\end{itemize}
Moreover by symmetry conditions described below, the matrix $A_0$ will necessarily be diagonalizable.
Therefore up to conjugation and translation, we may write
\begin{equation}\label{A equation}
A_{11} = \mxx{\lambda + d}{b+c}{b-c}{\lambda - d}\ \ \text{and}\ \ A_0 = \mxx{a}{0}{0}{-a}
\end{equation}
for some constants $a,b,c,d,$ and $\lambda$.
Also the value $A_{10}$ may be expressed explicitly in terms of the (normalized) first moment $B(0)$ of the weight matrix $W(x)$ by
\begin{equation}\label{A10 equation}
A_{10} = [A_0,B(0)]-A_{11}B(0),\ \ B(0) := \left(\int xW(x)dx\right)\left(\int W(x)dx\right)^{-1}.
\end{equation}
Note that reflection preserves the values of $A_{11}$ and $A_0$, while changing the sign of $B(0)$.
Thus up to translation and conjugation, we may assume that either $c=2-b$ or else $b=c=0$ and $B(0)$ is symmetric.

The $2\times 2$ hypergeometric matrix differential operator $\mathfrak D$ is completely determined by the value of $A_{11},A_0$ and $B(0)$, and $\mathfrak D$ in turn determines the corresponding weight matrix $W(x)$ up to similarity.
In this way, the  value of $(W(x),\mathfrak D)$ is determined by the values of $A_{11},A_0,$ and $B(0)$, up to similarity of $W(x)$.

The next theorem describes the classifying space of all hypergeometric matrix Bochner pairs in terms of an open subset of a real algebraic set in $\bbr^9$, parametrized by the values of $a,b,c,d,\lambda$ and $B(0)_{ij}$.

Note that we exclude the reducible matrix Bochner pairs, ie. those pair which are equivalent to the direct sum of two scalar pairs, since these are uninteresting.
\begin{thm}[Classifying Space]\label{classifying space}
Up to translation, conjugation and reflection, the classifying space $\mathcal E\subset\bbr^9$ of all irreducible $2\times 2$ real hypergometric matrix Bochner pairs is defined by the $9$-tuples $(a,b,c,d,\lambda,B(0)_{11},B(0)_{12},B(0)_{21},B(0)_{22})$ satisfying one of the following cases
\begin{enumerate}[(I)]
\item  $b = c = 0$ and $d = 0$;
\begin{align*}
(a-1)B(0)_{11} &= (a+1)B(0)_{22}\\
B(0)_{22}^2\lambda^2 &= \left((4a^2-\lambda^2)B(0)_{21} + 4\right)(a-1)^2\\
B(0)_{12} &= 1
\end{align*}
%\item  $c = 2-b$ and $b - 2 + d^2/4 = 0$
%\begin{align*}
%B(0)_{11} &= \dots\\
%B(0)_{22} &= \dots\\
%B(0)_{12} &= \dots
%\end{align*}
\item  $c = 2-b$, $b = 1$ and $d = -1/2$;
\begin{align*}
%B(0)_{11} &= \frac{-16a^2B(0)_{21} + 8aB(0)_{21}\lambda - 4a B(0)_{21} + 8a + 4B(0)_{21}\lambda - 2B(0)_{21} - 2\lambda + 1}{2\lambda-1},\\
B(0)_{11} &= (4a+2)B(0)_{21} - 1 + 8a\frac{1-2aB(0)_{21}}{2\lambda-1},\\
%B(0)_{22} &= \frac{-16a^2B(0)_{21} + 8aB(0)_{21}\lambda + 12a B(0)_{21} + 8a - 4B(0)_{21}\lambda - 2B(0)_{21} - 2\lambda - 5}{2\lambda+1},\\
B(0)_{22} &= (4a-2)B(0)_{21} - 1 + (8a-4)\frac{1-2aB(0)_{21}}{2\lambda+1},\\
%B(0)_{12} &= -\frac{4B(0)_{21}(-4a+2\lambda+1) + 8}{2\lambda+1}.
B(0)_{12} &= -4B(0)_{21}-8\frac{1-2aB(0)_{21}}{2\lambda+1}.
\end{align*}
\item  $c = 2-b$ and $2ad+(\lambda+1)(4(b-1)+d^2) = 0$;
\begin{align*}
B(0)_{11}d &= ((b-1)B(0)_{12}-B(0)_{21})\lambda - 2B(0)_{21}\\
B(0)_{22}d &= ((b-1)B(0)_{12}-B(0)_{21})\lambda + 2(b-1)B(0)_{12}\\
%B(0)_{12} &= -\frac{(d^2+2(b-1))}{2(b-1)^2}B(0)_{21}\pm \frac{d}{\lambda}\frac{\sqrt{(4(b-1) + d^2)\lambda^2B(0)_{21}^2 + 4(b-1)^2}}{2(b-1)^2}\\
%B(0)_{12} &= -\frac{(d^2+2(b-1))}{2(b-1)^2}B(0)_{21}\pm \frac{\sqrt{(d^2 + 2(b-1))^2B(0)_{21}^2 + 4(b-1)^2((d/\lambda)^2-B(0)_{21}^2)}}{2(b-1)^2}\\
 d^2 &= \lambda^2\frac{(B(0)_{21}+(b-1)B(0)_{12})^2}{1-\lambda^2B(0)_{12}B(0)_{21}}
\end{align*}
%\item  $c = 2-b$, $b = 1$ and $2a + (\lambda + 1)d = 0$
%\begin{align*}
%B(0)_{11} &= -B(0)_{21}(\lambda+2)/d\\
%B(0)_{22} &= -B(0)_{21}\lambda/d\\
%B(0)_{12} &= -\frac{B(0)_{21}}{d^2} + \frac{1}{B(0)_{21}\lambda^2}
%\end{align*}
\end{enumerate}
\end{thm}
Families (I) and (II) are three dimensional, and family (III) is four dimensional.
Furthermore, (I) can be seen as a limit of (III) if we conjugate the hypergeometric pairs in (III) by a diagonal matrix and then take an appropriate limit killing the off-diagonal terms of $A_{11}$.

Smooth paths in $\mathcal E$ define isospectral deformations of the weight matrix $W(x)$ preserving the bispectral property of having a $W$-symmetric hypergeometric matrix differential operator $\mathfrak D$.
When we can construct a bispectrality preserving isospectral deformations by other means, it will define natural parametrizations of subspaces of $\mathcal E$.
Many such examples arise from \emph{isomonodromic} deformations of the corresponding hypergeometric equation which fix the poles.
However, other natural bispectrality preserving isospectral deformations exist which do not preserve the action of the monodromy group.
For example, we may consider the deformation defined (for $c=2-b$) by
\begin{align}\label{translation deformation}
a\mapsto& \sqrt{4k^2(b-1) + (a + kd)^2}\\\nonumber
b\mapsto& \frac{a^2(b-1)}{4k^2(b-1) + (a + kd)^2} + 1\\\nonumber
d\mapsto& \frac{ad + 4k(b-1) + kd^2}{\sqrt{4k^2(b-1) + (a + kd)^2}}\\\nonumber
\lambda\mapsto& \lambda-2k\\\nonumber
B(0) \mapsto& U(k)B(k)U(k)^{-1}
\end{align}
where here $k$ is a deformation parameter and $U(k)$ is a matrix diagonalizing $A_{11}k + A_0$ such that $(U(k)A_{11}U(k)^{-1})_{12} = 2$.
The function $B(k)$ is a matrix-valued rational function of $k$ given by the degree zero component three-term recursion relation for the sequence of monic orthogonal matrix polynomials associated to $W(x)$ and has a natural expression in terms of an integral in $W(x)$, as described in \eqref{bc equation} below.
Despite the enigmatic expression of \eqref{translation deformation}, the deformation formulas above arise naturally from consideration of the $\Ad$-condition described below.

For practical applications, the explicit values of the hypergeometric matrix Bochner pairs parametrized by the classifying space described above are desireable.
These are provided explicitly in the next theorem.
\begin{thm}[Classification Theorem]\label{classification theorem}
Suppose that $(W(x),\mathfrak D)$ is a real hypergeometric matrix Bochner pair corresponding to a point $(a,b,c,d,\lambda, B(0)_{11}, B(0)_{12}, B(0)_{21},B(0)_{22})$ in the classifying space $\mathcal E$.
Then the corresponding weight matrix $W(x)$ is of the form
\begin{equation}\label{W and Q}
W(x) = Q(x)(1-x)^{-\lambda/2-1-\gamma/2 - \sigma_-/2}(1+x)^{-\lambda/2-1+\gamma/2 - \sigma_+/2}.
\end{equation}
for some symmetric, matrix-valued rational function $Q(x)$ and integers $\sigma_\pm\in \{0,1\}$.
Here
$$\gamma = \frac{1}{2}\tr(A_{11}B(0)) = \frac{1}{2}((\lambda + d)B(0)_{11} + (b+c)B(0)_{21} + (b-c)B(0)_{12} + (\lambda + d)B(0)_{22})$$
and in particular $\det(W(x)) = (1-x)^{-\lambda-2-\gamma}(1+x)^{-\lambda-2+\gamma}$.

Furthermore, for each of the families (I-III) defined above, we have the following expressions for the corresponding matrix Bochner pairs (up to similarity of $W(x)$)
\begin{enumerate}[(I)]
\item $\sigma_-=0$, $\sigma_+=0$, $\gamma = arB(0)_{22}$ and
\begin{align*}
Q(x)_{11} &= (1+Q(x)_{12}^2)/Q_{22}(x)\\
Q(x)_{12} &= \frac{1}{(1-x^2)s}(-arB(0)_{22} + arx -(r-2)x)\\
Q(x)_{22} &= \frac{1}{(1-x^2)s}\frac{-r^2aB(0)_{22}^2 + (a-1)(r(r-2)B(0)_{22}x + (r-2)(x^2+1)) + 4a}{(a-1)r+2a}
\end{align*}

%-(5*b22 - 2)*(5*b22 + 2)*(-5*a*b22 + 5*a*x - 3*x)**2/(3*(a - 1)*(7*a - 3)*(x - 1)**2*(x + 1)**2)  
%-(5*b22 - 2)*(5*b22 + 2)*(-25*a*b22**2 + 15*a*b22*x + 3*a*x**2 + 7*a - 15*b22*x - 3*x**2 - 3)**2/(3*(a - 1)*(7*a - 5)**2*(7*a - 3)*(x - 1)**2*(x + 1)**2)     

%-(7*b22 - 2)*(7*b22 + 2)*(-7*a*b22 + 7*a*x - 5*x)**2/(5*(a - 1)*(9*a - 5)*(x - 1)**2*(x + 1)**2)
%-(7*b22 - 2)*(7*b22 + 2)*(-49*a*b22**2 + 35*a*b22*x + 5*a*x**2 + 9*a - 35*b22*x - 5*x**2 - 5)**2/(5*(a - 1)*(9*a - 7)**2*(9*a - 5)*(x - 1)**2*(x + 1)**2)
\begin{align*}
\mathfrak D
  & = \partial_x^2(1-x^2)I + \partial_x\mxx{(x\lambda-B(0)_{22}\lambda\frac{a+1}{a-1}}{-2a-\lambda}{\frac{1}{2a+\lambda}\left(B(0)_{22}^2\frac{\lambda^2}{(a-1)^2}-4\right)}{\lambda x-B(0)_{22}\lambda} + \mxx{a}{0}{0}{-a}
\end{align*}
where here
$$r = \frac{\lambda}{a-1}\ \ \text{and}\ \ s^2 = \frac{(r-2)(a-1)((r+2)a-(r-2))}{4-r^2B(0)_{22}^2}.$$

\item $\sigma_-=0$, $\sigma_+=1$, $\gamma = r(1-2aB(0)_{21}) + \lambda$
\begin{align*}
Q(x)_{11} &= ((1+x)^{-1}+Q(x)_{12}^2)/Q_{22}(x)\\
Q(x)_{12} &= \frac{1/2}{(1-x^2)s}(4arB(0)_{21}-8a+2x\lambda+3x + 2\lambda + 3)\\
Q(x)_{22} &= \frac{1}{(1-x^2)s}(4arB(0)_{21}-4ax + 3x + 2\lambda + 3)
\end{align*}
%-b21*(-4*a + 2*y + 1)*(16*a**2*b21 - 8*a*b21*y - 4*a*b21 - 4*a*x - 4*a + 2*x*y + 3*x + 2*y + 3)**2/(16*a*(x - 1)**2*(x + 1)*(-4*a + 2*y + 3))
%-b21*(-4*a + 2*y + 1)*(16*a**2*b21 - 8*a*b21*y - 4*a*b21 - 8*a + 2*x*y + 3*x + 2*y + 3)**2/(4*a*(x - 1)**2*(x + 1)*(-4*a + 2*y + 3))
\begin{align*}
\mathfrak D
  & = \partial_x^2(1-x^2)I + \partial_x(x+1)\mxx{\lambda-1/2}{2}{0}{\lambda+1/2}\\
  & + \partial_x\mxx{(2a+1)rB(0)_{21}-4a}{-2rB(0)_{21}+4}{(r/2)B(0)_{21}}{(2a-1)rB(0)_{21}-4a + 2} + \mxx{a}{0}{0}{-a}
\end{align*}
where here
$$r = 4a-2\lambda-1\ \ \text{and} \ \ s^2 = \frac{4a(3-4a+2\lambda)}{B(0)_{21}r}.$$
%\item for $b = 1$ in case (III) with $s = (\lambda/d)B(0)_{21}$
%\begin{align*}
%W(x) &= \mxx{\left(1 + \frac{2x}{s} + \frac{1}{s^2}\right)/d^2}{-\left(1 + \frac{x}{s}\right)/d}{-\left(1 + \frac{x}{s}\right)/d}{1}W(x)_{22}\\
%W(x)_{22} &= (1-x)^{-3/2 - \lambda/2 - s(\lambda+1)/2}(1+x)^{-3/2 - \lambda/2 + s(\lambda+1)/2}
%\end{align*}
%\begin{align*}
%\mathfrak D &= \partial_x^2(1-x)^2I + \partial_xx\mxx{\lambda+d}{2}{0}{\lambda-d}\\
%&+\partial_x\mxx{s(d+\lambda+2)}{s(1 + 1/d) + s^{-1}(1 - 1/d)}{-s(1+1/d)}{s(\lambda-d)} + \mxx{-d(\lambda+1)/2}{0}{0}{d(\lambda+1)/2}
%\end{align*}
\item $\sigma_-=1$, $\sigma_+=1$, $\gamma = \frac{(\lambda+1)r(B(0)_{21}^2\lambda^2 + (b-1)(r^2-1))}{B(0)_{21}^2\lambda^2 - (b-1)(r^2-1)}$
\begin{align*}
Q(x)_{11} &= ((1-x^2)^{-1}+Q(x)_{12}^2)/Q_{22}(x)\\
Q(x)_{12} &= \frac{1}{s}(B(0)_{21}^2\lambda^2(x-r) + (b-1)(x+r)(1-r^2))\\
Q(x)_{22} &= \frac{B(0)_{21}\lambda}{s}(B(0)_{21}^2\lambda^2 + (b-1)(r^2+2xr+1))\\
\end{align*}
\begin{align*}
\mathfrak D
 & = \partial_x^2(1-x^2)I + \partial_xx\mxx{\frac{B(0)_{21}\lambda^2r-B(0)_{21}^2\lambda^2  + (b-1)(r^2-1)}{B(0)_{21}\lambda r}}{2}{2b-2}{\frac{B(0)_{21}\lambda^2r+B(0)_{21}^2\lambda^2  - (b-1)(r^2-1)}{B(0)_{21}\lambda r}}\\
 & + \partial_x\mxx{B(0)_{21}-\lambda r-2r}{-\frac{B(0)_{21}\lambda (r^2+1) - r(r^2-1)}{B(0)_{21}\lambda r}}{\frac{B(0)_{21}\lambda (B(0)_{21}\lambda -r)}{r}}{-\lambda (B(0)_{21}+r)}\\
 & + \partial_x(b-1)(1-r^2)\mxx{\frac{(-B(0)_{21}^2\lambda ^2 + 2B(0)_{21}r\lambda (\lambda +1))}{B(0)_{21}\lambda ((b-1)(1-r^2)+B(0)_{21}^2\lambda ^2}}{\frac{(r-1)(r+1)}{B(0)_{21}^2\lambda ^2r}}{\frac{r^2+1}{r(1-r^2)}}{\frac{B(0)_{21}^2\lambda ^2 + 2B(0)_{21}\lambda r(\lambda +1)}{B(0)_{21}\lambda ((b-1)(1-r^2) + B(0)_{21}^2\lambda ^2}}\\
 & + \partial_x\frac{-(b-1)^2(1-r^2)^2}{B(0)_{21}\lambda ((b-1)(1-r^2) + B(0)_{21}^2\lambda ^2}\mxx{1}{0}{0}{-1}\\
 & + \frac{(\lambda  + 1)((b-1)(r-1)^2 + B(0)_{21}^2\lambda ^2)((b-1)(r+1)^2 + B(0)_{21}^2\lambda ^2)}{2B(0)_{21}\lambda r((b-1)(1-r^2) + B(0)_{21}^2\lambda ^2)}\mxx{1}{0}{0}{-1}
\end{align*}
where here
$$r^2 = 1-\lambda^2B(0)_{12}B(0)_{21}\ \ \text{and}\ \ s^2 = (B(0)_{21}^2\lambda^2 + (b-1)(r-1)^2)(B(0)_{21}^2\lambda^2 + (b-1)(r+1)^2).$$
\end{enumerate}
\end{thm}

The $2\times 2$ hypergeometric matrix Bochner pairs given by three dimensional family (I) were previously completely obtained in \cite{calderon}, though they differ from the ones expressed here by similarity and an affine transformation $x\mapsto (x+1)/2$ and their derivation methods are brute force and do not extend past their paper.
It's worth remarking that our methods extend to all the $2\times 2$, including the Hermite and Laguerre-type families.
Furthermore, they may be extended to a classification of the $N\times N$ \emph{theoretically}, barring limits in memory and computing power.
Even earlier, specific examples were obtained directly through representation theory \cite{grunbaum2002,koelink}.
Subcollections of the three dimensional families (I) and (II) were also found in \cite{grunbaum2003a}.
To our knowledge, no members of the the four dimensional family (III) have appeared in the literature.

\section{Background}
\subsection{The matrix Bochner problem}
A sequence of (monic) orthogonal matrix polynomials for a weight matrix $W(x)$ is a unique sequence of $N\times N$-matrix valued polynomials $P(x,0),P(x,1),\dots$ with $P(x,n)$ monic of degree $n$ for each $n$, satisfying the property that $\int P(x,n)W(x)P(x,m)^*dx = 0$ for $m\neq n$.
Sequences of orthogonal matrix polynomials were first investigated more than $70$ years ago by Krein \cite{krein} in the context of the matrix moment problem and Hermitian operators with nontrivial deficiency index.
Since then, orthogonal matrix polynomials have proven useful in a variety of areas of both pure and applied mathematics, including spectral theory, quasi-birth and death processes, signal processing, Gaussian quadrature, special functions, random matrices, integrable systems, and representation theory.

Of particular interest are the orthogonal matrix polynomials which are simultaneously eigenfunctions of a second order matrix-valued differential operator, since they naturally generalize the three families of \vocab{classical orthogonal polynomials}: the  Hermite, Laguerre, and Jacobi whose ubiquity throughout many fields of mathematics is well known.
However, it is worth noting that the value of the matrix-valued analogs of the classical orthogonal polynomials is largely based on their potential as a natural generalization.
A great deal of work remains to be done to translate this potential into concrete applications.

Bochner \cite{bochner} proved that up to affine transformation the only orthogonal polynomials on the real line which are eigenfunctions of a second-order differential operator are the three families of classical orthogonal polynomials.
For orthogonal matrix polynomials, however, the world is far richer and a complete classification has proved elusive.
The $N\times N$ matrix Bochner problem, posed originally by Dur\'{a}n \cite{duran1997}, is the problem of classifying all sequences of $N\times N$ orthogonal matrix polynomials $P(x,n)$ satisfying a second-order differential equation
\begin{equation}\label{main problem equation}
\left(\frac{d^2}{dx^2}P(x,n)\right)a_2(x) + \left(\frac{d}{dx}P(x,n)\right)A_1(x) + P(x,n)A_0(x) = \Lambda(n)P(x,n)
\end{equation}
for some sequence of matrices $\Lambda(0),\Lambda(1),\dots\in M_N(\bbr)$ and functions $a_2(x),A_1(x),A_0(x)$.
The coefficients in the differential equation \eqref{main problem equation} are polynomials of the form
$$a_2(x) = a_{22}x^2 + a_{21}x + a_{20},\ \ A_1(x) = A_{11}x + A_{10},\ \ \text{and}\ \ A_0(x) = A_0,$$
for some constants $a_{2i}\in \bbr$ and $A_{1i},A_0\in M_N(\bbr)$.
The precise value of $\Lambda(n)$ is determined explicitly by comparing leading coefficients of polynomials in \eqref{main problem equation}, giving
\begin{equation}\label{lambda equation}
\Lambda(n) = a_{22}n^2 + (A_{11}-a_{22}I)n + A_0.
\end{equation}
Alternatively, we can rewrite this in terms of a matrix-valued differential operator acting on the right as
\begin{equation}\label{main problem equation operator form}
P(x,n)\cdot\mathfrak D = \Lambda(n)P(x,n),\ \ \text{where}\ \ \mathfrak D = \partial_x^2a_2(x)I + \partial_xA_1(x) + A_0.
\end{equation}
More generally, one may consider the algebra of matrix-valued differential operators
$$\mathcal D(W) := \left\lbrace\mathfrak D = \sum_{m=0}^d \partial_x^mA_m(x): \forall n\exists\Lambda(n)\in M_N(\bbc)\ \text{s.t.}\ P(x,n)\cdot\mathfrak D = \Lambda(n)P(x,n)\right\rbrace$$
and ask for which values of $W(x)$ the above algebra is nontrivial.

By a result of Tirao and Gr\"unbaum \cite{grunbaum2007b}, the differential operator from \eqref{main problem equation operator form} can be taken to be symmetric with respect to $W(x)$, so that
\begin{equation}
\int (P(x)\cdot\mathfrak D)W(x)Q(x)^*dx = \int P(x)\cdot (Q(x)\cdot\mathfrak D)^*dx,\ \ \forall\ P(x),Q(x)\in M_N(\bbr(x)).
\end{equation}
Moreover, any differential operator $\mathfrak D$ of the form \eqref{main problem equation operator form} which is $W(x)$-symmetric in this fashion automatically satisfies \eqref{main problem equation operator form} for $\Lambda(n)$ given by \ref{lambda equation}.
We call the pair $(W(x),\mathfrak D)$ with $\mathfrak D$ a $W(x)$-symmetric differential operator a \vocab{matrix Bochner pair}.
Thus the $N\times N$ matrix Bochner problem may be thought of as the problem of classifying all $N\times N$ matrix Bochner pairs.

Very recently, the author and Yakimov \cite{casper2018} combined methods from integrable systems, algebraic geometry, bispectrality, and the representation theory of semiprime PI algebras to obtain a general classification for $N>1$ under mild hypotheses, in terms of noncommutative bispectral Darboux transformations of direct sums of classical weights.
In particular, the authors proved in the $2\times 2$ case that when the algebra $\mathcal D(W)$ is noncommutative the weight matrix $W(x)$ must come from a noncommutative bispectral Darboux transformation of a classical weight.
In this paper we carefully explore the $\Ad$-condition (see \eqref{ad condition 0} below) in order to \emph{explicity} classify the $2\times 2$ matrix weights with hypergeometric matrix differential operators in $\mathcal D(W)$.

\subsection{Bispectrality and the $\Ad$-condition}
The $N\times N$ matrix Bochner problem is a \vocab{bispectral problem} in the sense of \cite{bakalov}, as the (monic) orthogonal matrix polynomials are also eigenfunctions of a difference equation in the spectral parameter $n$ (valid for $n\geq 0$ with $C(0)=0$)
\begin{equation}
P(x,n+1) + B(n)P(x,n) + C(n)P(x,n-1) = xP(x,n).
\end{equation}
This equation is commonly referred to as the three-term recursion relation for $P(x,n)$ and is an immediate consequence of the pairwise orthogonality of the polynomials $P(x,n)$.
The sequences $B(n)$ and $C(n)$ are defined by
\begin{align}\label{bc equation}
B(n) = \left(\int xP(x,n)W(x) P(x,n)^* dx\right)M(n)^{-1},\ \ \ C(n) = M(n)M(n-1)^{-1}
\end{align}
where here $M(n)$ defines the norm squared value of $P(x,n)$ with respect to $W(x)$, ie.
\begin{equation}\label{m equation}
M(n) = \int P(x,n)W(x) P(x,n)^* dx.
\end{equation}
Note that the definitions of $M(n)$, $B(n)$, and $C(n)$ from \eqref{bc equation} and \eqref{m equation} above imply that $M(n) = C(n)C(n-1)\dots C(1)M(0)$ and that $M(n)$, $B(n)M(n)$, $\Lambda(n)M(n)$, $C(n)M(n)$, and $C(n+1)M(n)$ are all symmetric.

In terms of matrix-valued difference operators acting on the left, we can rewrite the three-term recursion relation as
\begin{equation}
\mathscr L\cdot P(x,n) = xP(x,n),\ \ \mathscr L = I\mathscr S + B(n) + C(n)\mathscr S^*,
\end{equation}
where here $\mathscr S$ is the standard shift operator $(\mathscr S\cdot P)(x,n) = P(x,n+1)$ and $\mathscr S^*$ is its adjoint.
From this point of view, the algebra $\mathcal D(W)$ is the algebra of bispectral differential operators analogous to Wilson's construction in the classical bispectral situation \cite{wilson}.
Bispectrality in particular implies that $\mathscr L$ and $\Lambda(n)$ satisfy the $\Ad$-condition
\begin{equation}\label{ad condition 0}
[\mathscr L,[\mathscr L,\Lambda(n)]] = 2a_2(\mathscr L)
\end{equation}
where here $[\cdot,\cdot]$ denotes the usual commutator bracket and we are viewing $\Lambda(n)$ as an operator acting on sequences of matrices by left multiplication.

\section{The $\Ad$-condition}
\subsection{Fundamental Calculations}
A sequence of monic orthogonal matrix polynomials $\{P(x,n)\}$ associated with an $N\times N$ matrix Bochner pair $(W,\mathfrak D)$ is bispectral in that
\begin{align}
\mathscr L\cdot P(x,n) &= P(x,n)x,\\
P(x,n)\cdot\mathfrak D &= \Lambda(n)P(x,n).
\end{align}
where here $\mathscr L$ is the difference operator
\begin{equation}
\mathscr L = \mathscr S + B(n) + C(n)\mathscr S^*
\end{equation}
obtained from  the $3$-term recursion relation for the $P(x,n)$.
The operators $\mathscr S$ and $\mathscr S^*$ are the forward shift operator and its adjoint which act on $N\times N$ matrix-valued functions $F(n): \bbn\rightarrow M_N(\bbc)$ by $(\mathscr S\cdot F)(n) = F(n+1)$, and $(\mathscr S^*F)(n) = F(n-1)$ with $F(-1):=0$.

The bispectral property of orthogonal matrix polynomials leads to a certain relationship between the coefficients of $\mathscr L$ and $\Lambda(n)$, known as the $\Ad$-condition.
For any difference operators $\mathscr M,\mathscr N$, we recurseively define
$$\Ad_{\mathscr M}^{k+1}(\mathscr N) = \Ad_{\mathscr M}(\Ad_{\mathscr M}^k(\mathscr N)),\ \ \text{with}\ \ \Ad_{\mathscr M}(\mathscr N) = \mathscr M\mathscr N - \mathscr N\mathscr M.$$
Then the $\Ad$-condition says
$$\Ad_{\mathscr L}^3(\Lambda(n)) = 0.$$
This follows from bispectrality along with the fact that $\Ad_x^3(\mathfrak D) = 0$.
In the special case that $\mathfrak D$ has scalar leading coefficient, ie:
\begin{equation}
\mathfrak D = \partial_x^2a_2(x)I + \partial_x A_1(x) + A_0
\end{equation}
for some $a_2(x) = a_{22}x^2 + a_{21}x + a_{20}$, $A_1(x) = A_{11}x + A_{10}$ with $a_{2i}\in\bbr$ and $A_{11},A_{10},A_0\in M_2(\bbr)$
the $\Ad$-condition can be reduced further.
In this case, $\Ad_x^2(\mathfrak D) = 2a_2(x)$ and therefore
$$2a_2(\mathscr L)\cdot P(x,n) = P(x,n)\cdot 2a_2(x) = P(x,n)\cdot \Ad_x^2(\mathfrak D) = \Ad_{\mathscr L}^2(\Lambda(n))\cdot P(x,n).$$
As a consequence, we obtain the (reduced) $\Ad$-condition
\begin{equation}\label{ad condition}
\Ad_{\mathscr L}^2(\Lambda(n)) = 2a_2(\mathscr L),\ \ \mathscr L = \mathscr S + B(n) + C(n)\mathscr S^*.
\end{equation}
Recall that $M(n) = \|P(x,n)\|_W^2$.
If we choose a factorization $\wt{\mathscr L} = M(n)^{-1/2}\mathscr L M(n)^{1/2}$ for each $n$ and define $\wt \Lambda(n) = M(n)^{-1/2}\mathscr L M(n)^{1/2}$, then $\wt{\mathscr L}$ represents a three-term recursion relation for a sequence of normalized orthogonal polynomials for $W(x)$, and in particular is symmetric.
The associated operator $\wt{\mathscr O} = [\wt{\mathscr L},\wt \Lambda(n)]$ is skew-symmetric and satisfies the \vocab{string equation}
\begin{equation}
[\wt{\mathscr L},\wt{\mathscr O}] = 2a_2(\wt{\mathscr L}).
\end{equation}

Note that $\mathscr S^*$ is not quite the inverse of $\mathscr S$, since $F(n) - (\mathscr S^*\mathscr S\cdot F)(n)= \delta_{n,0}F(0)$.
Thus $\mathscr S^*$ is only a right inverse of $\mathscr S$ and $\mathscr S$ is not quite unitary.
To fix this issue, we can extend the domain of the sequences under consideration to functions $F(n): \bbz\rightarrow \bbc$, so that the shift operators act by $(\mathscr SF)(n) = F(n+1)$ and $(\mathscr SF)(n) = F(n-1)$, respectively.
In this case $\mathscr S$ is unitary as an operator on $M_2(\bbc)\otimes \ell^2(\bbz)$.
For the remainder of the paper, our shift operators will always act explicitly on sequences with domain $\bbz$ and we will search for pairs $(\mathscr L,\Lambda(n))$ satisfying the $\Ad$-condition.

Equation \ref{ad condition} expands into
$$Z_2(n)\mathscr S^2 + Z_1(n)\mathscr S + Z_0(n) + Z_{-1}(n)\mathscr S^* + Z_{-2}(n)\mathscr S^{*2} = 0,$$
for some sequences $Z_i(n)$ whose values may be obtained explicitly via direct calculation.
Thus the $\Ad$-condition simply says $Z_i(n) = 0$ for all $n\geq 0$ and all $i=-2,\dots, 2$.

The vanishing condition for the $Z_i(n)$'s is repetitive in the sense that the expression for $Z_i(n)$ is zero if and only if $Z_{-i}(n)$ is zero because
\begin{align*}
Z_i(n)\|P(x,n+i)\|_W^2
  & = \langle[\Ad_{\mathscr L}^2(\Lambda(n))-2a_2(\mathscr L)]\cdot P(x,n),P(x,n+i)\rangle_W\\
  & = \langle P(x,n)\cdot [\Ad_x^2(\mathfrak D)-2a_2(x)I],P(x,n+i)\rangle_W\\
  & = \langle P(x,n),P(x,n+i)\cdot [\Ad_x^2(\mathfrak D)-2a_2(x)I]\rangle_W\\
  & = \langle P(x,n),[\Ad_{\mathscr L}^2(\Lambda(n))-2a_2(\mathscr L)]\cdot P(x,n+i)\rangle_W\\
  & = \|P(x,n)\|_W^2Z_{-i}(n)^*.
\end{align*}
Here we used the fact that $\mathfrak D$ is self-adjoint with respect to $W(x)$.
Furthermore, the value of $Z_2(n)$ evaluates to being identically zero, and therefore the vanishing condition reduces to the pair of equations $Z_1(n) = 0$ and $Z_0(n) = 0$.
Explicitly:
\begin{align}\label{B update equation}
0 &= -2a_{22}(n+1)B(n+1) + B(n+1)(A_{11}(n+1)+A_0) - (A_{11}(n+2)+A_0)B(n+1)\\\nonumber
       &+ 2a_{22}(n-1)B(n) + (A_{11}n+A_0)B(n) - B(n)(A_{11}(n-1)+A_0) - 2a_{21}I\\
\label{C update equation}
0 &= -2a_{22}(2n+1)C(n+1) + C(n+1)(A_{11}n+A_0) - (A_{11}(n+2)+A_0)C(n+1)\\\nonumber
       &+ 2a_{22}(2n-3)C(n) + (A_{11}n+A_0)C(n) - C(n)(A_{11}(n-2)+A_0)\\\nonumber
       &+ B(n)^2(A_{11}n+A_0) -2 B(n)(A_{11}n+A_0)B(n) + (A_{11}n+A_0)B(n)^2\\\nonumber
       &-2a_{22}B(n)^2-2a_{21}B(n)-2a_{20}I
\end{align}
We refer to these below as the \vocab{update equations for $B(n)$ and $C(n)$}, respectively.

\subsection{Linearization of the update equations}
The update equations for $B(n)$ and $C(n)$ allow us to determine $B(n+1)$ and $C(n+1)$ from $B(n)$ and $C(n)$.
In particular they show that for any choice of $\Lambda(n)$ and starting data $B(0)$, $C(1)$ there exist values $B(n)$ and $C(n)$ satisfying the $\Ad$-condition.
Furthermore, outside of certain exceptional situations, the sequences $B(n)$ and $C(n)$ are completely determined by $\Lambda(n)$ and the initial values $B(0)$ and $C(1)$.
The update equations in particular are linear, and by expressing them formally as linear systems we can obtain a closed-form solution for $B(n)$ in terms of the initial data.
The explicit expression in particular proves that, outside of certain exceptional cases, $B(n)$ is equal to a rational function in $n$ for almost every $n$.

To begin, let $\vec b(n)$ and $\vec c(n)$ vectors whose entries are the values of $B(n)$ and $C(n)$, in standard order and let $H(n) = nH_1 + H_0$ and $K(n) = nK_1 + K_0$ be the $N^2\times N^2$ matrices whose entries are defined by
\begin{align}
(H_1)_{(i-1)N+j,(k-1)N+\ell}(n) &= 2a_{22}\delta_{i,k}\delta_{j,\ell} - \delta_{i,k}(A_{11})_{\ell,j} + \delta_{j,\ell}(A_{11})_{i,k}\\
(H_0)_{(i-1)N+j,(k-1)N+\ell}(n) &=-2a_{22}\delta_{i,k}\delta_{j,\ell} + \delta_{i,k}(A_{11} - A_0)_{\ell,j} + \delta_{j,\ell}(A_0)_{i,k}\\
(K_1)_{(i-1)N+j,(k-1)N+\ell}(n) &= 4a_{22}\delta_{i,k}\delta_{j,\ell} - \delta_{i,k}(A_{11})_{\ell,j} + \delta_{j,\ell}(A_{11})_{i,k}\\
(K_0)_{(i-1)N+j,(k-1)N+\ell}(n) &=-6a_{22}\delta_{i,k}\delta_{j,\ell} + \delta_{i,k}(2A_{11}-A_0)_{\ell,j} + \delta_{j,\ell}(A_0)_{i,k}
\end{align}

Using these expressions, the update equations may be rewritten as
\begin{align}
H(n+2)\vec b(n+1) &= H(n)\vec b(n) - 2a_{21}\vec \sigma,\\
K(n+2)\vec c(n+1) &= K(n)\vec c(n) + \vec \theta(n)
\end{align}
where here $\vec \sigma$ and $\vec \theta(n)$ are vectors whose entries are
\begin{align*}
\vec\sigma_{iN+j} &= \delta_{ij}\\
\vec \theta(n)_{iN + j}
  & = + (B(n)^2(A_{11}n+A_0) -2 B(n)(A_{11}n+A_0)B(n) + (A_{11}n+A_0)B(n)^2)_{ij}\\
       &-2(a_{22}B(n)^2+a_{21}B(n)+a_{20}I)_{ij}
\end{align*}

The update equation for $\vec b(n)$ turns out to have a simple closed-form solution.
\begin{lem}\label{HK lemma}
Suppose that $H(n)$ is nonsingular for all integers $n\geq \ell$.
Then
\begin{equation*}
\vec b(n) = H(n)^{-1}H(\ell)H(n+1)^{-1}\left[H(\ell+1)\vec b(\ell) -2a_{21}(n-\ell)H\left(\frac{n+\ell+1}{2}\right)H(\ell)^{-1}\vec\sigma\right].
\end{equation*}
\end{lem}
\begin{proof}
Note by assumption $H_\ell := H(\ell)$ is nonsingular.
For all $m\geq 0$
\begin{align*}
H(m+n+\ell)^{-1}H(m+\ell)
  & = H(m+n+\ell)^{-1}(H(m+n+\ell)-nH_1)\\
  & = I - nH(m+n+\ell)^{-1}H_1\\
  & = I - n[(n+m)(H_\ell^{-1}H_1) + I]^{-1}(H_\ell^{-1}H_1)
\end{align*}
This expression is a function of the single matrix $H_\ell^{-1}H_1$ and shows that $H(m+n+\ell)^{-1}H(m+\ell)$ and $H(m'+n'+\ell)^{-1}H(m'+\ell)$ commute for all $m,m',n,n'$ with $m+n\geq 0$ and $m'+n'\geq 0$.

The update equation for $\vec b(n)$ shows inductively that for all $n > \ell$
\begin{align*}
\vec b(n)
  & = \left(\prod_{k=\ell}^{n-1} H(k+2)^{-1}H(k)\right)\vec b(\ell)\\
  & - 2a_{21}\sum_{j=\ell}^{n-1} \left(\prod_{k=0}^{n-j-1} H(n+1-k)^{-1}H(n-1-k)\right)H(j)^{-1}\vec\sigma
\end{align*}
where the empty product is interpreted as the identity.
Note that by the commutativity relation above, the order in the product is not important.
Using this commutativity relation, the products telescope leading to the equation
\begin{align*}
\vec b(n)
  & = H(n)^{-1}H(\ell)H(n+1)^{-1}H(\ell+1)\vec b(\ell)\\
  & - 2a_{21}H(n)^{-1}H(\ell)H(n+1)^{-1}\left(\sum_{j=\ell}^{n-1} H(j+1)\right) H(\ell)^{-1}\vec\sigma,
\end{align*}
Evaluating the sum, we obtain the equation stated in the lemma.
\end{proof}
The values of the $B(n)$'s do not depend at all on the values of the $C(n)$'s.
Moreover, the above implies that the entries of $B(n)$ must be rational in $n$.
\begin{cor}\label{B is rational}
Suppose that $H(n)$ is nonsingular for all integers $n\geq \ell$.
The sequence $B(n)$ is a rational function of $n$ for all $n\geq\ell$.
\end{cor}
\begin{proof}
This is clear from the formula for $B(n)$ determined above.
\end{proof}
The values of the $C(n)$'s depend on the values of both the $B(n)$'s and $C(n)$'s, with the $C(n)$-update equation in particular depending quadratically on $B(n)$.
For this reason, a simple expression for $C(n)$ is not immediately apparent, nor is it clear whether $C(n)$ must also be rational or not.
Our strategy in finding the value of $C(n)$ below will instead rely on inferences determined by certain symmetry conditions described in the next section.

\subsection{Symmetry conditions}
In addition to the update formulas for $B(n)$ and $C(n)$, the values of $B(n)$ and $C(n)$ satisfy certain symmetry conditions.
These symmetry conditions in turn impose restrictions on the allowed values of $B(0),C(1),A_{11},A_0,$ and $a_2(x)$.

To begin, let $M(n) = \langle P(x,n),P(x,n)\rangle_W$ for $n\geq 0$.
Note that $M(n)$ is positive definite and symmetric for all $n$.
\begin{lem}\label{lambda symmetry}
For all $n\geq 0$, the value of $\Lambda(n)M(n)$ is $*$-symmetric.
\end{lem}
\begin{proof}
Since $(W,\mathfrak D)$ is a Bochner pair, we know in particular that $\mathfrak D$ is $W$-symmetric.
Therefore
$$\Lambda(n)M(n) = \langle P(x,n)\cdot\mathfrak D, P(x,n)\rangle_W = \langle P(x,n), P(x,n)\cdot\mathfrak D\rangle_W = M(n)\Lambda(n)^*.$$
\end{proof}

\begin{lem}\label{B symmetry}
For all $n\geq 0$, the value of $B(n)M(n)$ is $*$-symmetric and satisfies
\begin{equation}
B(n)M(n)= \langle x P(x,n),P(x,n)\rangle_W
\end{equation}
\end{lem}
\begin{proof}
Since $xP(x,n) = P(x,n+1) + B(n)P(x,n) + C(n)P(x,n-1)$ we see that
$$B(n)M(n) = \langle xP(x,n),P(x,n)\rangle_W = \langle P(x,n), xP(x,n)\rangle_W = M(n)B(n)^*.$$
\end{proof}

\begin{lem}\label{C symmetry}
For all $n\geq 0$ we have $M(n) = C(n)C(n-1)\dots C(1)M(0)$ and both $C(n+1)M(n)$ and $C(n)M(n)$ are $*$-symmetric.
\end{lem}
\begin{proof}
Since $xP(x,n) = P(x,n+1) + B(n)P(x,n) + C(n)P(x,n-1)$ we see that for all $n\geq 1$
$$C(n)M(n-1) = \langle xP(x,n),P(x,n-1)\rangle_W = \langle P(x,n), xP(x,n-1)\rangle_W = M(n).$$
Therefore $C(n) = M(n)M(n-1)^{-1}$ and our formula for $M(n)$ follows immediately.
Since $M(n)$ is $*$-symmetric, we get that $C(n+1)M(n) = M(n+1)$ is $*$-symmetric for all $n$.
This in particular implies that $C(n)M(n-1) = M(n-1)C(n)^*$ and therefore
$$C(n)M(n) = C(n)C(n)M(n-1) = C(n)M(n-1)C(n)^*.$$
This latter expression is $*$-symmetric, and therefore $C(n)M(n)$ is $*$-symmetric.
\end{proof}

Note that if we multiply the $C(n)$-update equation on the right by $M(n)$ and then take the symmetric part
we obtain
\begin{align}\label{M update equation}
 &2a_{22}(2n+1)M(n+1) + (A_{11}M(n+1)+M(n+1)A_{11}^*)\\\nonumber
  & = M(n)\left(2a_{22}(2n-3)M(n-1)^{-1} + M(n-1)^{-1}A_{11} + A_{11}^*M(n-1)^{-1}\right)M(n)\\\nonumber
  & + [B(n),[B(n),\Lambda(n)]]M(n) - 2a_2(B(n))M(n).
\end{align}
We call this the \vocab{update equation for $M(n)$}.
Note that unlike the update equations for $B(n)$ and $C(n)$, the update equation for $M(n)$ is nonlinear.

\section{Isospectral deformations}
In the case of a scalar weight $\rho(x)$ on $\bbr$, isospectral deformations of $\rho(x)$ correspond to deformations of the form
\begin{equation}\label{simple deformation}
\rho(x;\vec t) = \rho(x)\exp\left(\sum_{n=1}^\infty t_nx^n\right)
\end{equation}
where here $\vec t=(t_1,t_2,\dots)$ is the deformation parameter.
Such a deformation is isospectral in the sense that the three-term recursion relation describing the orthogonal polynomials of $\rho(x;\vec t)$
$$\mathscr L_{\vec t}\cdot p(x,n;\vec t) = xp(x,n;\vec t),\ \ \mathscr L_{\vec t} = \mathscr S + a(n;\vec t) + b(n;\vec t)\mathscr S^*$$
preserves the spectral parameter $x$.
Furthermore, the deformation equations for $a(n;\vec t)$ and $b(n;\vec t)$ are determined by solutions of the infinite Toda hierarchy
$$\frac{\partial}{\partial t_n}\mathscr L_{\vec t} = \frac{1}{2}[(\mathscr L^n)_s,\mathscr L],$$
satisfying certain symmetries insuring that the deformed values $\mathscr L_{\vec t}$ remain having only three terms \cite{adler}.
Here for any shift operator
$$\mathscr M = \sum_{k=1}^\ell a_k(n)\mathscr S^k + a_0(n) + \sum_{k=1}^{\infty}a_{-k}(n)(\mathscr S^*)^k,$$
the expression $\mathscr M_s$ denotes the skew symmetric operator
$$\mathscr M = \sum_{k=1}^\ell (a_k(n)\mathscr S^k + a_k(n-k)(\mathscr S^*)^k).$$

In the matrix case, we mimic this same definition and consider isospectral deformations of $W(x)$ as deformations of the corresponding three-term recursion relation preserving the spectral parameter
$$\mathscr L_{\vec t}\cdot P(x,n;\vec t) = xP(x,n;\vec t),\ \ \mathscr L_{\vec t} = \mathscr S + A(n;\vec t) + B(n;\vec t)\mathscr S^*.$$
This is more complicated than the scalar case, due to the existence of \emph{noncommutative} isospectral deformations, ie. deformations $W(x;\vec t)$ which do not commute with $W(x;0)$, so that $W(x;\vec t)$ is not a scalar multiple of $W(x)$.
Furthermore, a suitably general replacement of the Toda hierarchy and its relevant symmetries must still be worked out.
This is further complicated by the fact that in this paper, we are interested specifically in \emph{bispectral} isospectral deformations, ie. isospectral deformations which preserve the property of the existence of a polynomial $\Lambda(n,\vec t) = -n^2I + \Lambda_1(\vec t)n + \Lambda_0(\vec t)$ satisfying the $\Ad$-condition
$$\Ad_{\mathscr L_{\vec t}}^2(\Lambda(n;\vec t)) = 2(I-\mathscr L_{\vec t}^2).$$
In such a case, the associated orthogonal polynomials $P(x,n;k)$ will satisfy a second-order matrix-valued hypergeometric differential equation
$$(1-x^2)\frac{d^2}{dx^2}P(x,n;k) + \frac{d}{dx}P(x,n;k)A_1(x;k) + P(x,n;k)A_0(k) = P(x,n;k)\Lambda(n;k).$$
As such, the deformations that we construct here will be ad-hoc and without the benefit of a general theoretical framework.
Indeed, every smooth curve in the classifying space $\mathcal E$ of Theorem \ref{classifying space} defines a bispectrality-preserving isospectral deformation, so a simple theoretical framework allowing for the explicit construction of all such deformations would solve our original problem immediately.

The simplest example of an isospectral deformation preserving the $\Ad$-condition is the translation
\begin{align*}
\mathscr L_k &= \mathscr SI + B(n+k) + C(n+k)(\mathscr S)^*,\\
\Lambda(n;k) &= \Lambda(n+k).
\end{align*}
After conjugating and translating the associated matrix Bochner pair to get $A_{11}$ and $A_0$ in the form \eqref{A equation}, we obtain the isospectral deformation described by \eqref{translation deformation}.

More examples of isospectral deformations preserving the $\Ad$-condition arise from isomonodromic deformations of the hypergeometric equation satisfied by $W(x)$ fixing the poles, as we next describe.
\subsection{Isomonodromic deformations}
The pair $(W(x),\mathfrak D)$ is a matrix Bochner pair if and only if $W(x)$ is a weight matrix and satisfies the differential equation
\begin{equation}\label{fundamental differential equation}
(W(x)a_2(x))'' - (W(x)A_1(x)^*)' + WA_0^* = A_0W
\end{equation}
along with the two boundary conditions that
\begin{equation}\label{boundary conditions}
a_2(x)W(x)\rightarrow 0,\ \ \text{and}\ \ (a_2(x)W(x))' - A_1(x)W(x)\rightarrow 0\ \ \text{as}\ \ x\rightarrow\partial\supp(W).
\end{equation}
In particular if $(W(x),\mathfrak D)$ is a hypergeometric matrix Bochner pair, then $W(x)$ is itself a solution of a hypergeometric equation.

Setting the two boundary conditions aside for now, the equation \eqref{fundamental differential equation} naturally decomposes into symmetric and skew-symmetric components.
The symmetric component is the noncommutative Pearson equation
\begin{equation}\label{Pearson equation}
2(W(x)a_2(x))' = A_1(x)W(x) + W(x)A_1(x)^*.
\end{equation}
The skew-symmetric component reduces to the following ordinary differential equation we will refer to as the \vocab{auxillary equation}
\begin{equation}\label{aux equation}
(W(x)A_1(x)^*-A_1(x)W(x))' = 2(WA_0^*-A_0W)
\end{equation}
To summarize, we have the following theorem
\begin{thm}[\cite{duran2004}]\label{differential condition}
A pair $(W(x),\mathfrak D)$ is a matrix Bochner pair if and only if it satisfies the noncommutative Pearson equation \eqref{Pearson equation}, the auxillary equation \eqref{aux equation}, and the boundary conditions \eqref{boundary conditions}.
\end{thm}
Note that since the weight matrix $W(x)$ is symmetric, the above system of equations is over-determined and only in special instances can we find a solution of the noncommutative Pearson equation which also satisfies the auxillary equation.

The solution space of the Pearson equation \eqref{Pearson equation} is always $N^2$-dimensional and generically the subspace of solutions also satisfying the auxillary equation \eqref{aux equation} is trivial.
In special instances, however, this subspace is nontrivial and proper.
When the zeros of $\det(A_1(x))$ are away from $-1$ and $1$, this also means that the action of the monodromy group of \eqref{Pearson equation} will preserve this subspace, forcing the corresponding $N^2$-dimensional representation of the monodromy group to be reducible.

In the two-dimensional case, the subspace of solutions of \eqref{Pearson equation} satisfying \eqref{aux equation}, when nontrivial, is typically one dimensional.
This forces the action of the monodromy group to act by scalars on $W(x)$, as is reflected in the explicit form of our expressions for $W(x)$ found in Theorem \ref{classification theorem}.

Generically we expect the representation corresponding to the monodromy group of \eqref{Pearson equation} to be irreducible, so the fact that weights belong to the class where it reduces strongly suggests isomonodromic deformations as a natural candidate for transforming hypergeometric matrix Bochner pairs into new ones.

In the one-dimensional case, the weight is (up to a constant) $(1-x)^\alpha(1+x)^\beta$ and there are no interesting isomonodromic deformations, because both the solution and the monodromy action are determined by the exponents $\alpha$ and $\beta$.
However, in the matrix case interesting isomonodromic deformations actually exist.
This is hinted at from the form of $W(x)$ in Theorem \ref{classification theorem}, ie. $W(x) = (1-x)^\alpha(1+x)^\beta Q(x)$ for some matrix-valued rational function $Q(x)$.
The action of the monodromy group again is determined by $\alpha$ and $\beta$.
However, the classifying space provides three and four-dimensional spaces of solutions and therefore contains smooth paths fixing the exponents $\alpha$ and $\beta$ and therefore the monodromy action.

The equation \eqref{Pearson equation} may be rephrased in terms of a resonant Fuchsian equation, and \emph{some} of the isomonodromic deformations of this latter equation define bispectrality-preserving isospectral deformations of $W(x)$.
A more explicit description of which isomonodromic deformations define such bispectrality preserving isospectral deformations has been considered by the author and will be the topic of future work.

\section{The two-dimensional case}
\subsection{The exceptional cases}
For the rest of the paper, we consider exclusively the case when $W(x)$ is a $2\times 2$ weight matrix.
In this case, we can take $\Lambda(n)$ to have the form specified by \eqref{main problem equation operator form} with $A_{11}$ and $A_0$ of the form \eqref{A equation} in the introduction, ie.
\begin{equation*}
A_{11} = \mxx{\lambda+d}{b+c}{b-c}{\lambda-d},\ \ A_0 = \mxx{a}{0}{0}{-a}.
\end{equation*}
Our ultimate goal is to determine the values of $a,b,c,d,\lambda$ and $B(0)$ for which the $B(n)$ and $C(n)$-update equations provide a solution of the $\Ad$-condition \eqref{ad condition}.
If $H(n)$ is nonsingular for infinitely many values of $n$ (equivalently for almost every value of $n$), then there will exist an integer $\ell>0$ such that $B(n)$ is determined from $B(\ell)$ for all $n\geq\ell$ by Lemma \ref{HK lemma}.
Moreover, the value of $B(n)$ determines the value of $M(n)$ up to a scalar, and these values in turn must satisfy the $M(n)$ update equation.
However, the update equation for $M(n)$ is complicated and nonlinear and for a generic choice of the values of $a_{22}$, $A_{11}$, $A_0$ and $B(\ell)$, the associated values of $B(n)$ contradict the $M(n)$-update equation.

Inspired by this, we are lead to consider specifically the case when the $4\times 4$ matrices $H(n)$ are singular for all $n$, so that $B(n)$ is not determined from some value $B(\ell)$ for all $n\geq \ell$.

In the $2\times 2$ case, the value of $H(n)$ is given by
\begin{align*}
H(n) &=
\left(\begin{array}{cccc}
2a_{22} & c-b        & c+b        & 0\\
-(c+b)  & 2a_{22}+2d & 0          & c+b\\
-(c-b)  & 0          & 2a_{22}-2d & c-b\\
0       &-(c-b)      & -(c+b)     & 2a_{22}
\end{array}\right)n\\
&+
\left(\begin{array}{cccc}
y-2a_{22} + d & -(c-b)          & 0              & 0\\
c+b           & y-2a_{22}+2a-d  & 0              & 0\\
0             & 0               & y-2a_{22}-2a+d & -(c-b)\\
0             & 0               & c+b            & y-2a_{22}-d
\end{array}\right)
\end{align*}
The next lemma shows that for generic $n$ the matrix $H(n)$ is nonsingular, except in very special cases.
\begin{lem}[Exceptional cases]\label{exceptional lemma}
The determinant of $H(n)$ is nonzero for almost every $n$, unless one of the following conditions holds
\begin{enumerate}[(i)]
\item  $a_{22} = 0$, $\lambda^2 = b^2-c^2+d^2$, $a=0$
\item  $a_{22} = 0$, $b=\pm c$, $\lambda = d$,
\item  $a_{22} = 0$, $b=\pm c$, $\lambda =-d$,
\item  $a_{22} = d$, $b =\pm c$, $\lambda = d+2a$
\item  $a_{22} =-d$, $b =\pm c$, $\lambda =-(d+2a)$
\end{enumerate}
\end{lem}
\begin{proof}
This follows from explicit calculation of the determinant of $H(n)$.
\end{proof}
Throughout this section we will refer to the above conditions \vocab{exceptional cases}.
We will call $\Lambda(n)$ \vocab{exceptional} if we are in such a situation.

Since we are interested in the hypergeometric case when $a_{22} = -1$, the only cases we will be concerned with are (iv) and (v).
These constitute specifically four different possibilities:
\begin{enumerate}
\item[(iv.a)] $a_{22}=d$, $b=c$, $\lambda=d+2a$
\item[(iv.b)] $a_{22}=d$, $b=-c$, $\lambda=d+2a$
\item[(v.a)] $a_{22}=-d$, $b=c$, $\lambda=-(d+2a)$
\item[(v.b)] $a_{22}=-d$, $b=-c$, $\lambda=-(d+2a$)
\end{enumerate}
Furthermore since $c = b = 0$ or $c = 2-b$, we can reduce to the case of (iv.a) and (v.a).
Thus up to similarity, we can reduce to the two cases (iv.a) and (v.a) with $c=b=1$.

\subsection{Obtaining $M(n)$ from $B(n)$}
Furthermore, we may exploit the low dimensionality to deduce explicit dependencies between the matrices $B(n),M(n)$ and $\Lambda(n)$.
\begin{prop}\label{getting M from B}
Let $n\geq 0$.
Suppose that $[B(n),\Lambda(n)]\neq 0$ and $M(n)$ is not a scalar.
Then there exists a constant $\beta(n)$ such that
\begin{equation}\label{B to M}
M(n) = \beta(n)[B(n),\Lambda(n)]J
\end{equation}
where here $J$ is the symplectic matrix $J = \mxx{0}{1}{-1}{0}$.
\end{prop}
\begin{remk}
It is worth noting that this is the precise point in the derivation that we require $M(n)$, $B(n)$ and $\Lambda(n$) to all be real-valued.
\end{remk}
\begin{proof}
The space of $2\times 2$ skew-symmetric matrices is one dimensional.
Therefore there exist constants $\gamma_1,\gamma_2\in\bbr$ satisfying
$$B(n)-B(n)^* = \gamma_1J,\ \ \Lambda(n)-\Lambda(n)^* = \gamma_2J.$$

Next note that the matrices $M(n), M(n)^2, B(n)M(n)$, and $\Lambda(n)M(n)$ are all symmetric by Lemma \ref{lambda symmetry} and Lemma \ref{B symmetry}.
Therefore they are linearly dependent and there exist constants $\alpha_1,\alpha_2,\alpha_3,\alpha_4\in\bbr$, not all zero, satisfying
$$\alpha_1M(n) + \alpha_2 M(n)^2 +  \alpha_3B(n)M(n) + \alpha_4\Lambda(n)M(n) = 0.$$
Since $M(n)$ is invertible, this implies
$$\alpha_1I + \alpha_2 M(n) +  \alpha_3B(n) + \alpha_4\Lambda(n) = 0.$$
Taking the commutator of this expression with $B(n)$ or $\Lambda(n)$ and using the fact that $B(n)M(n) = M(n)B(n)^*$, and $\Lambda(n)M(n) = M(n)\Lambda(n)^*$ we get
$$\alpha_2\gamma_1M(n)J =-\alpha_4[B(n),\Lambda(n)]$$
$$\alpha_2\gamma_2M(n)J = \alpha_3[B(n),\Lambda(n)]$$
Since $M(n)$ is not a scalar, either $\alpha_3\neq 0$ or $\alpha_4\neq 0$.
If $\alpha_3\neq 0$, then the above shows $\alpha_2\gamma_2\neq 0$ and therefore we obtain \eqref{B to M} with $\beta(n) = \alpha_3/(\alpha_2\gamma_2)$.
Alternatively, if $\alpha_4\neq 0$ we obtain \eqref{B to M} with $\beta(n) = -\alpha_4/(\alpha_2\gamma_1)$.
\end{proof}
Consequently for sequential values of $n$ satisfying the assumptions of the previous proposition, we have the following equation for the $C(n)$'s
\begin{equation}\label{B to C}
C(n+1) = M(n+1)M(n)^{-1} = \frac{\beta(n+1)}{\beta(n)}[B(n+1),\Lambda(n+1)][B(n),\Lambda(n)]^{-1}
\end{equation}
The formula for $M(n)$ in \eqref{B to M} has the property that the symmetry conditions from Lemma \ref{lambda symmetry}, Lemma \ref{B symmetry}, and Lemma \ref{C symmetry} are \emph{automatically satisfied}.
\begin{lem}
Suppose that there exist constants $\beta(n-1),\beta(n),\beta(n+1)$ such that
$$M(j) = \beta(j)[B(j),\Lambda(j)]J,\ \ j=n-1,n,n+1,$$
and also
$$C(j+1) = \frac{\beta(j+1)}{\beta(j)}[B(j+1),\Lambda(j+1)][B(j),\Lambda(j)]^{-1},\ \ j= n-1,n.$$
Then the matrices $M(n)$, $B(n)M(n)$, $C(n)M(n)$, $C(n+1)M(n)$, and $\Lambda(n)M(n)$ are all symmetric.
\end{lem}
\begin{proof}
For any trace-free $2\times 2$ matrix $X$, the product $XJ$ is symmetric.
In particular $M(n)$ is symmetric, as are the matrices
$$B(n)M(n) = \beta(n)B(n)[B(n),\Lambda(n)]J = \beta(n)[B(n),B(n)\Lambda(n)]J$$
and
$$\Lambda(n)M(n) = \beta(n)\Lambda(n)[B(n),\Lambda(n)]J = \beta(n)[\Lambda(n)B(n),\Lambda(n)]J$$
are both symmetric.
Also clearly $C(n+1)M(n) = M(n+1)$ and $C(n)M(n) = M(n)M(n-1)^{-1}M(n)$ are both symmetric.

\end{proof}

We can also demonstrate that in the case $\Lambda(n)$ and $B(n)$ commute, the weight matrix $W(x)$ is reducible and therefore not interesting from the point of view of classification.
\begin{lem}\label{Lambda not scalar}
Suppose $(W(x),\mathfrak D)$ is a normalized $2\times 2$ matrix Bochner pair and that $\Lambda(n)$ is scalar-valued.
Then $W(x)$ is reducible.
\end{lem}
\begin{proof}
Recall that $B(0)$ is symmetric and that $A_{10} = [B(0),A_0]-A_{11}B(0)$.
Since $\Lambda(n)$ is scalar-valued, we know that $A_0$ and $A_{11}$ are scalar-valued and $A_{10}$ is symmetric.
Thus up to similarity, we can take $B(0)$ and $A_{10}$ to both be diagonal.
Hence $\mathfrak D$ is also diagonal and the pair $(W(x),\mathfrak D)$ is reducible.
\end{proof}

\begin{lem}\label{B does not commute diag}
Suppose $(W(x),\mathfrak D)$ is a normalized $2\times 2$ matrix Bochner pair.
If $\Lambda(n)$ is diagonal and $B(n)$ commutes with $\Lambda(n)$ for all $n$, then $W(x)$ is reducible.
\end{lem}
\begin{proof}
Without loss of generality, we may assume that $\Lambda(n)$ is not scalar-valued.
Then since $B(n)$ commutes with $\Lambda(n)$ for all $n$, we know that $B(n)$ is diagonal for all $n$.
Thus the matrices $B(0), A_{11},$ and $A_0$ are all diagonal.
Hence $\mathfrak D$ is also diagonal and the pair $(W(x),\mathfrak D)$ is reducible.
\end{proof}

\begin{prop}\label{B does not commute}
Suppose $(W(x),\mathfrak D)$ is a normalized $2\times 2$ hypergeometric matrix Bochner pair.
If $B(n)$ commutes with $\Lambda(n)$ for all $n$, then $W(x)$ is reducible.
\end{prop}
\begin{proof}
By the previous two lemmas, without loss of generality we may assume that $\Lambda(n)$ is not diagonal.
Without loss of generality, we may assume that $\Lambda(n)$ is not scalar-valued, and therefore $a\neq0$ and also either $b\neq 0$ or $c\neq 0$.
The $B(n)$-update equation reduces to
$$(A_{11}-2(n+1)I)B(n+1) = B(n)(A_{11}-2(n-1)).$$
Now since $B(n)$ commutes with $\Lambda(n)$ and the latter is not diagonal, we can write
$$B(n) = \alpha(n)(A_{11}n + A_0) + \beta(n)I$$
for some constants $\alpha(n),\beta(n)\in\bbc$.
Inserting this above and taking the commutator with $A_{11}$ and simplifying we find
$$\frac{\alpha(n)}{\alpha(n+1)}I=J^{-1}(A_{11}-2(n+1)I)J(A_{11}-2(n-1)I)^{-1}$$
where here $J = \mxx{0}{1}{-1}{0}$.
For almost every $n$, the matrix on the right is not scalar-valued so this is a contradiction.
\end{proof}

Thus our task is reduced to finding $B(n)$ and $\beta(n)$ satisfying the update equation for $B(n)$ with $M(n)$ defined as in \eqref{B to M} satisfying the update equation for $M(n)$.

\subsection{Algebraic equations}
Let $\wt M(n) = [B(n),\Lambda(n)]J$ so that $M(n)$ and $\wt M(n)$ are scalar multiples of one another.
The update equation for $M(n)$ implies that the three symmetric matrices
$$2a_{22}(2n+1)\wt M(n+1) + (A_{11}\wt M(n+1)+\wt M(n+1)A_{11}^*),$$
$$\wt M(n)\left(2a_{22}(2n-3)\wt M(n-1)^{-1} + \wt M(n-1)^{-1}A_{11} + A_{11}^*\wt M(n-1)^{-1}\right)\wt M(n),$$
$$[B(n),[B(n),\Lambda(n)]]\wt M(n) - 2a_2(B(n))\wt M(n),$$
will all be linearly dependent.
Consequently, the determinant of the $3\times 3$ matrix formed by taking the diagonal and superdiagonal entries of the above matrices as rows must be zero.
Inserting the explicit equation for $B(n)$ from Lemma \ref{HK lemma} directly into the result, we obtain a polynomial in $n$ whose coefficients are polynomials in $a,d,y$, and $B(0)_{ij}$ for $1\leq i\leq j \leq 2$.
The expression for this polynomial is extremely large and it's value is difficult to manipulate with computer software, let alone by hand.
For this reason, we omit the expression here.
As a polynomial in $n$, the expression has degree $46$ and leads to a system of $46$ polynomial equations in the seven variables $a,d,y,$ and $B(0)_{ij}$ from which we can obtain the existence, or lack thereof, of solutions via Gr\"obner basis techniques.
Furthermore, when the determinant vanishes the scalars that describe the linear dependence of the above three quantities must satisfy a recursion relation associated with the three-term recurrence relation.
As a result, we obtain precisely the families in the statement of Theorem \ref{classifying space}.
The explicit values of $W(x)$ are then solved for by solving the associated noncommutative Pearson equation \eqref{Pearson equation} and finding the corresponding solution which also satisfies the auxillary equation \eqref{aux equation}.
This results in the family of $2\times2$ hypergeometric matrix Bochner pairs in the statement of Theorem \ref{classification theorem}.

\section{Exclusion of the exceptional cases}
In the previous section, we established the families of Jacobi matrices as the only possible families, outside of the exceptional situation wherein the matrix $B(n)$ does not necessarily have a rational expression.
Our remaining task in order to prove the classification theorem stated in the introduction is to exclude this exceptional case, showing that there are no $2\times2$ hypergeometric matrix Bochner pairs when conditions (iv) or (v) of Lemma \ref{exceptional lemma} hold.
\subsection{Analysis of case (iv)}
In this section, we consider the exceptional cases, ie. (iv) or (v) from Lemma \ref{exceptional lemma}.
We begin with the analysis of case (iv), which up to similarity is defined by $b=c$, $d=a_{22}$ and $\lambda = d + 2a$.
In this case
$$A_{10} = \mxx{-2(a+a_{22})B(0)_{11}-2cB(0)_{21}}{-2(2a+a_{22})B(0)_{21}-2cB(0)_{22}}{0}{-2aB(0)_{22}},$$
$$A_{11} = \mxx{2a+2a_{22}}{2c}{0}{2a},\ \ \text{and}\ \ A_0 = \mxx{a}{0}{0}{-a}.$$
Here we have used the fact that $B(0)$ is symmetric ($B(0)_{12} = B(0)_{21}$) and that
$$A_{10} = [B(0),A_0]-A_{11}B(0).$$
The associated weight matrix can be obtained by starting with the noncommutative Pearson equation
$$2(a_2(x)W(x))' = A_1(x)W(x) + W(x)A_1(x)^*.$$
In our particular case, this implies
\begin{align*}
(a_2(x)W(x)_{11})' &= -2(B(0)_{11}(a+a_{22}) + B(0)_{12}c - x(a+a_{22}))W(x)_{11}\\
                   &+ -2(aB(0)_{12} + B(0)_{12}(a+a_{22}) + B(0)_{22}c - cx)W(x)_{12}\\
(a_2(x)W(x)_{12})' &= (-a(B(0)_{22} - x) - B(0)_{11}(a + a_{22}) - B(0)_{12}c + x(a + a_{22}))W(x)_{12}\\
                   &+ (-2aB(0)_{12} - a_{22}B(0)_{12} - B(0)_{22}c + cx)W(x)_{22}\\
(a_2(x)W(x)_{22})' &= 2a(x-B(0)_{22})W(x)_{22}
\end{align*}

This system of equations may be solved for $W(x)$ \emph{explicitly}, by solving for $W(x)_{22}$, then $W(x)_{12}$, and finally for $W(x)_{11}$.

Additionally, we know that the weight matrix $W(x)$ must satisfy the second order differential equation
$$(a_2(x)W(x))'' - (W(x)A_1(x)^*)' + W(x)A_0^* = A_0W(x)$$
along with the usual boundary conditions that
$$a_2(x)W(x),\ \ \text{and}\ \ (a_2(x)W(x))' - A_1(x)W(x)$$
vanish at the endpoints of the support of $W(x)$.
We show that the only such solutions in the Jacobi case are reducible solutions.
\begin{prop}[Jacobi case (iv)]
Suppose that $(W(x),\mathfrak D)$ is a normalized $2\times 2$ matrix Bochner with $a_2(x) = 1-x^2$, $b=c$, $d=-1$ and $\lambda = 2a-1$.
Then $c = 0$ and $a = 1/2$ so that $\mathfrak D$ is diagonal and the matrix Bochner pair is reducible.
\end{prop}
\begin{proof}
Assume that $c\neq 0$.
Then up to similarity, we can take $c = 1$.
The noncommutative Pearson equation immediately tells us that up to a nonzero constant multiple
$$W(x)_{22} = (1-x)^{aB(0)_{22} - a - 1}(x + 1)^{-aB(0)_{22} - a - 1}.$$
Furthermore the upper right component of the noncommutative Pearson equation says
$$((1-x^2)W(x)_{12})' = f(x)((1-x^2)W(x)_{12}) + g(x)W(x)_{22}$$
where $f(x)$ and $g(x)$ are the rational functions given by
\begin{align*}
f(x) &= \frac{(a(B(0)_{11} + B(0)_{22}) + B(0)_{12}-B(0)_{11} + x(1-2a))}{1-x^2}\\
g(x) &= ((1- 2a)B(0)_{12} -B(0)_{22} + x)W(x)_{22}.
\end{align*}
However, if we combine this with the upper-right entry of the the general second order equation and simplify, we find also that
$$((1-x^2)W(x)_{12})' = \wt f(x)((1-x^2)W(x)_{12}) + \wt g(x)W(x)_{22}$$
where $\wt f(x)$ and $\wt g(x)$ are the rational functions given by
\begin{align*}
\wt f(x) &= \frac{2x(a(B(0)_{11}-B(0)_{22}) - B(0)_{11} + B(0)_{12} + x) + (1-2a)(1-x^2)}{(1-x^2)(a(B(0)_{22}-B(0)_{11})+B(0)_{11}-B(0)_{12}-x)}\\
\wt g(x) &= \frac{4a^2B(0)_{12}(B(0)_{22}-x) - 2aB(0)_{12}(B(0)_{22}+x) + 2aB(0)_{22}^2 - 4aB(0)_{22}x}{a(B(0)_{11}-B(0)_{22}) -B(0)_{11}+B(0)_{12} + x}\\
         &+ \frac{2ax^2 + 2B(0)_{12}x - 2B(0)_{22}x + x^2 + 1}{a(B(0)_{11}-B(0)_{22}) -B(0)_{11}+B(0)_{12} + x}
\end{align*}
Thus we have found two different first order linear equations which $W(x)_{12}$ satisfies.
One verifies directly that for no values of $B(0)_{ij}$ and $a$ does there exist $W(x)_{12}$ which can satisfy both of these second-order linear equations simultaneously.
Thus $c=0$.

Now take $c =0$.
By a similar argument to the previous paragraph, we obtain a pair of first order linear equations (omitted for brevity) which have a solution only if $a = 1/2$.
This completes the proof.
\end{proof}
Note that the case $c=0$ is also a special case of exceptional case (v), and will be handled in the next section.

\subsection{Analysis of case (v)}
Next we analyze case (v), which is defined up to similarity by $b=c$, $d=-a_{22}$, and $\lambda = -(d+2a)$.
In this case
$$A_{10} = \mxx{2aB(0)_{11}-2cB(0)_{21}}{-2cB(0)_{22}}{2(2a-a_{22})B(0)_{21}}{2(a-a_{22})B(0)_{22}},$$
$$A_{11} = \mxx{-2a}{2c}{0}{-2a+2a_{22}},\ \ \text{and}\ \ A_0 = \mxx{a}{0}{0}{-a}.$$
Unlike case (iv), the matrix $A_1(x)$ is no longer upper triangular so we will have trouble with solving the noncommutative Pearson equation directly.
Instead in this case, we solve the $\Ad$-condition directly for the values of $B(n)$.

For case (v), the update equation for $B(n)$ defines an overdetermined system of discrete update equations for $B(n)_{11},B(n)_{21},$ and $B(n)_{22}$, while imposing no condition on the value of $B(n)_{12}$.
Specifically, the upate equation tells us
$$\wt H(n+2) [B(n+1)_{11}\ B(n+1)_{21}\ B(n+1)_{22}]^T = \wt H(n) [B(n)_{11}\ B(n)_{21}\ B(n)_{22}]^T -2a_{21}[1\ 0\ 1]^T,$$
for $\wt H(n)$ the $3\times 3$ matrix
$$\wt H(n) = \left(\begin{array}{ccc}
2a_{22}n-2a-2a_{22} & 2cn & 0\\
0  & -2cn+2c & 2a_{22}n-2a\\
0  & 4a_{22}n -4a-2a_{22} & 0
\end{array}\right).$$
along with the condition that
\begin{align}\label{overdetermined eqn}
& (-2(n+2)c+2c)B(n+1)_{11} + 2(n+2)cB(n+1)_{22}\\\nonumber
  & = (-2nc+2c)B(n)_{11} + 2ncB(n)_{22}.
\end{align}

We temporarily ignore \eqref{overdetermined eqn} which overdetermines the system, and find values of $B(n)_{ij}$ satisfying the remaining system of update equations.
Proceeding as in the proof of Lemma \ref{HK lemma}, we can determine explicit equations for $B(n)_{11}$, $B(n)_{21}$ and $B(n)_{22}$ in terms of the initial values, directly from the condition involving the matrix $\wt H(n)$.
The explicit equations for $B(n)_{11}$, $B(n)_{21}$ and $B(n)_{22}$ depend on the value of $a_2(x)$ and are given in the Jacobi cases by\\
\begin{align*}
B(n)_{11}  &= B(0)_{11}\frac{a(a-1)}{(a + n)(a + n - 1)}\\
           &+ B(0)_{21}c\frac{(2a-1)(4a-1)n^2 + (2a-1)(4a^2-2a+1)n}{(a + n)(a + n - 1)(2a + 2n - 1)(2a + 2n + 1)}\\
B(n)_{21}  &= B(0)_{21}\frac{4a^2-1}{4a^2+8an + 4n^2-1}\\
B(n)_{22}  &= B(0)_{22}\frac{a(a + 1)}{(a + n)(a + n + 1)}\\
           &- B(0)_{21}c\frac{(2a+1)(4a+1)n^2 + (2a+1)(4a^2+2a+1)n}{(a + n)(a + n + 1)(2a + 2n - 1)(2a + 2n + 1)}
\end{align*}

Now we remember that the values of $B(n)_{ij}$ must also satisfy the \eqref{overdetermined eqn}.
Using this, we obtain certain constraints on the available values of $B(0)_{ij}$, $a$ and $c$.
Specifically, we find that one of the following holds
\begin{enumerate}[(\text{v.}a)]
\item $c = 0$
\item $c\neq 0$, $a^2 = 1$, $B(0)_{21} = 0$, and
$$a(B(0)_{11} - B(0)_{22}) - (B(0)_{11}+B(0)_{22})=0$$
\item $c\neq 0$, $B(0)_{11} + 2cB(0)_{21} - B(0)_{22} = 0$, and
$$2a(B(0)_{11}-B(0)_{22}) + B(0)_{11}+B(0)_{22} = 0.$$
\end{enumerate}

In each of these cases, the differential criteria imply that there exists no solution.

\section{Bispectral Darboux transformations}
In this section, we obtain some of the families of orthogonal matrix polynomials in the Classification Theorem \ref{classification theorem} by factoring pairs of $1\times 1$ matrix Bochner pairs.
Specifically, we start with a $2\times 2$ matrix Bochner pair of the form $(R(x),\diag(\mathfrak d_1,\mathfrak d_2))$ for $(r_i(x),\mathfrak d_i)$ a $1\times 1$ matrix Bochner pair and $R(x) = \diag(r_1(x),r_2(x))$.
Then we look for factorizations of the form
$$\mxx{\mathfrak d_1}{0}{0}{\mathfrak d_2} = (\partial_x S(x) + C)(\partial_x T(x) + F),$$
where here $S(x) = Ax + B$ and $T(x) = \pm \adj(S(x))$ for $A,B,C,F\in M_2(\bbr)$.
Each such factorization constitutes a noncommutative Darboux transformation.
In the special case that we have an additional symmetry condition
$$\partial_xS(x) + C = R(x)(-T(x)^*\partial_x + F(x)^*)S^*(x)R^{-1}(x)S(x)a_2(x)^{-1},$$
the Darboux transformation is bispectral.
In this case, the weight matrix
$$W(x) = \frac{1}{a_2(x)}T(x)\mxx{r_1(x)}{0}{0}{r_2(x)}T(x)^*$$
and the differential operator
$$\mathfrak D = (\partial_x T(x) + F)(\partial_x S(x) + C)$$
define an irreducible $2\times 2$ matrix Bochner pair.
Each of the families in the Classification Theorem \ref{classification theorem} is of this form.

The transformations described by the previous paragraph are all noncommutative bispectral Darboux transformations as defined in \cite{casper2018}, but of the special form used in \cite{casper2017}.
In particular, not all noncommutative bispectral Darboux transformations specified in \cite{casper2018} are of this form, since the latter includes transformations deriving from factorizations of operators with order higher than $2$.

\subsection{Factoring}
We begin with a classical Bochner pair
$$(\wt W(x),\wt{\mathfrak D}),\ \ \text{with}\ \ \wt{\mathfrak D} = \mxx{\mathfrak d_1}{0}{0}{\mathfrak d_2}$$
where here $\mathfrak d_1$ and $\mathfrak d_2$ are classical second-order operators, ie.
\begin{align*}
\mathfrak d_i = \partial_x^2a_2(x) + \partial_x(\alpha_i x + \beta_i) + \gamma_i
\end{align*}
where $\alpha_i,\beta_i,\gamma_i\in\bbr$ and $a_2(x) \in\{ 1,x,1-x^2\}$.
Note that the leading coefficient $a_2(x)$ is assumed to be the same for both operators in order to make $\mathfrak D$ and its Darboux transformations have scalar leading coefficients.

The Darboux transformations we consider are obtained from factorizations of $\wt{\mathfrak D}$ as a product of two first-order matrix differential operators
$$\wt{\mathfrak D} = (\partial_x S(x) + C)(\partial_x T(x) + F),$$
with $\pm T(x) = \adj(S(x))$ the adjugate of $S(x)$ and with $S(x) = Ax + B$ for constant matrices $A,B,C,F$ satisfying $\det(S(x)) = \pm a_2(x)$.

Our task is simplified by using the fact that for $2\times 2$ matrices, the adjugate satisfies
$$X + \adj(X) = \tr(X)I.$$
Since the adjugate is involutive, this implies
\begin{equation}\label{trace trick}
X\adj(Y) + Y\adj(X)  = \tr(Y\adj(X))I.
\end{equation}

Now consider our factorization of $\wt{\mathfrak D}$
\begin{align*}
\wt{\mathfrak D} &= (\partial_x S(x) + C)(\partial_x T(x) + F)\\
  & = \partial_x^2 S(x)T(x) + \partial_x(S'(x)T(x) + CT(x) + S(x)F) + CF\\
  & = \partial_x^2 a_2(x)I + \partial_x[x(\det(A)I + C\adj(A) + AF) + A\adj(B) + C\adj(B) + BF] + CF.
\end{align*}
We substitute
$$F = \pm(\adj(A) + \adj(C) + G).$$
and use \eqref{trace trick} to obtain
\begin{align*}
CF &= \pm(\det(C) + C(G + A)),\\
\pm C\adj(A) + AF &= \pm(\det(A)I + \tr(A\adj(C))I + AG)\\
\pm(A\adj(B) + C\adj(B)) + BF &= \pm(\tr(A\adj(B))I + \tr(B\adj(C))I + BG).
\end{align*}
For $\wt{\mathfrak D}$ to be diagonal, each of these expressions must also be diagonal, and therefore we must require that $C(G+A), AG,$ and $BG$ to all be diagonal.

Using this factorization, we obtain a new differential operator
\begin{align*}
\mathfrak D &= (\partial_x T(x) + F)(\partial_x S(x) + C)\\
  & = \partial_x^2 T(x)S(x) + \partial_x(T'(x)S(x) + FS(x) + T(x)C) + FC\\
  & = \partial_x^2 a_2(x)I \pm \partial_xx((\tr(\adj(A)C)+2\det(A))I + GA)\\
  & + \partial_x(2\adj(A)B + GB + \tr(\adj(C)B)I) \pm (\adj(A)+G)C \pm \det(C)I
\end{align*}

Summarizing the discussion above, we have the following lemma.
\begin{lem}\label{factorization lemma}
Fix a nonzero polynomial $a_2(x)$ of degree at most two.
Let $S(x) = Ax+B$ and $T(x) = \pm \adj(S(x))$, and $\det(S(x)) = \pm a_2(x)$.
Then for any $C$ and $G_0$ satisfying the condition that
$$C(G+A),\ \ AG,\ \ \text{and}\ \ BG\ \ \text{are all diagonal}$$
we have a factorization
$$(\partial_x (S(x) + C)(\partial_x T(x) + \adj(C) + G + A) = \mxx{\mathfrak d_1}{0}{0}{\mathfrak d_2}$$
where
$$\mathfrak d_i = \partial_x^2a_2(x) + \partial_x(\alpha_ix +\beta_i) + \gamma_i,$$
\begin{align}
\alpha_i &= \pm(2\det(A) + \tr(A\adj(C)) + [AG]_{ii})\\
\beta_i &= \pm(\tr(B\adj(C)) + \tr(A\adj(B)) + [BG]_{ii})\\
\gamma_i &= \pm(\det(C) + [C(G+A)]_{ii})
\end{align}
\end{lem}

As a consequence, we have the following theorem providing a large family of noncommutative bispectral Darboux transformations of classical weights.
\begin{thm}\label{darboux theorem}
Using the notation and assumptions of Lemma \ref{factorization lemma}, suppose that the operators $\mathfrak d_1$ and $\mathfrak d_2$ from the lemma are classical operators for classical weights $r_1(x)$ and $r_2(x)$, respectively.
Let $R(x) = \diag(r_1(x),r_2(x))$, $\mathfrak U = \partial_xS(x) + C$ and $\mathfrak V = \partial_x T(x) + F$.
If the concomitant of the operator $R(x)\mathfrak V^*$ vanishes at the endpoints of the support of $R(x)$ and the additional symmetry condition
$$\mathfrak Ua_2(x)^{-1}T(x)R(x)T^*(x) = R(x)\mathfrak V^*$$
is satisfied, then
$$W(x) = \frac{1}{a_2(x)}T(x)\mxx{r_1(x)}{0}{0}{r_2(x)}T(x)^*,$$
\begin{align*}
\mathfrak D
  & = (\partial_x T(x) + \adj(C) + G + A)(\partial_x S(x) + C)\\
  & = \partial_x^2a_2(x)I + \partial_x\left(x[2\det(A)I+\tr(AC)I + G\adj(A)] + (2A+G)\adj(B)+\tr(BC)I\right)\\
  & + \det(C)I + (G+A)C
\end{align*}
defines a Bochner pair $(W(x),\mathfrak D)$ and $W(x)$ is a noncommutative bispectral Darboux transformation of $\text{diag}(r_1(x),r_2(x))$.
In particular, the sequence of orthogonal matrix polynomials for $P(x,n)$ satisfies
\begin{equation}\label{p equation}
(An+C)P(x,n) := \mxx{p_1'(x,n)}{0}{0}{p_2'(x,n)}S(x) + \mxx{p_1(x,n)}{0}{0}{p_2(x,n)}C
\end{equation}
for $p_i(x,n)$ the sequence of monic orthogonal polynomials of $r_i(x)$.
\end{thm}
\begin{proof}
Let $\wt P(x,n) = \diag(p_1(x,n),p_2(x,n))$ for $p_i(x,n)$ the sequence of orthogonal polynomials for $r_i(x)$ and choose $\wt\Lambda(n)$ such that
$$\wt P(x,n)\cdot \wt{\mathfrak D} = \wt\Lambda(n)\wt P(x,n).$$
Set $\Lambda(n) = (An+C)^{-1}\wt\Lambda(n)(An+C)$.
The polynomials $P(x,n)$ defined as in the statement of the theorem satisfy the differential operator equation
\begin{align*}
P(x,n)\cdot\mathfrak D = (An+C)^{-1}\wt P(x,n)\cdot \mathfrak U\mathfrak V\mathfrak U = \Lambda(n)P(x,n).
\end{align*}
Moreover, they satisfy the orthogonality condition
\begin{align*}
(An+C)&\left(\int P(x,n) W(x) P(x,m)^* dx\right)(An+C)^*\\
&= \int (\wt P(x,n)R(x))\cdot \mathfrak V^* (\wt P(x,m)\cdot \mathfrak U)^* dx\\
&= \int \wt P(x,n)R(x)(\wt P(x,m)\cdot \mathfrak D)^* dx\\
&= \int \wt P(x,n)R(x)\wt P(x,m)^*dx \Lambda(m)^* = 0\ \ \text{for $m\neq n$}.
\end{align*}
Since $P(x,n)$ is monic of degree $n$ for each $n$, it defines a sequence of orthogonal polynomials for $W(x)$.
\end{proof}
By taking bispectral Darboux transformations of this form, we can obtain all of the matrices in family (I) and (II).
We suspect that the matrices in family (III) may be obtained more complicated forms of bispectral Darboux transformations of nonclassical weights.

\section*{Acknowledgement}
The author thanks F. Alberto Gr\"unbaum, Erik Koelink, Pablo Roman, and Boris Shapiro  for helpful discussions during a visit to Berkeley, Nijmegen, and Stockholm during 2019.
The author also thanks them and additionally Ignacio Zurri\'an for some helpful remarks on an earlier draft of this paper.
The author's travel was funded by an 2018 AMS-Simons Travel Grant.

\bibliographystyle{plain}
\bibliography{hyper}

\end{document}